\documentclass[10pt]{article}
\usepackage{amsmath,amssymb,amsthm,mathrsfs}
\theoremstyle{definition}
\newtheorem{theorem}{Theorem}[section]

\newtheorem{lemma}[theorem]{Lemma}

\newtheorem{corollary}[theorem]{Corollary}
\newtheorem{example}[theorem]{Example}
\DeclareMathOperator*{\ilim}{ind~ \! lim}

\title{ Imbedding Exotic Hida-Kubo-Takenaka Spaces into usual Hida disributions}
\author{Kei Harada \\
\begin{small}
Graduate School of Mathematics, Nagoya University 
\end{small} 
\\
\begin{small}
E-mail : kei.harada@math.nagoya-u.ac.jp
\end{small}
}
\date{}

\begin{document}

\maketitle
\begin{abstract}
We show that $\Gamma (\mathcal{N}_p) $, a subspace of exotic Hida-Kubo-Takenaka space is  naturally imbedded into $ (E)^\ast $, the usual 
Hida-Kubo-Takenaka space under some conditions. We also study on Heat Equations associated with exotic Laplacians, such as L\'{e}vy Laplacian. 
\end{abstract}

\section{Inroduction}
An infinite dimensional Laplacian introduced by P. L\'{e}vy, the so called L\'{e}vy Laplacian, has attracted  many scientists. 
Studying L\'{e}vy Laplacian is important in infinite dimensional analysis, 
because L\'{e}vy Laplacian is inherent to infinite dimensional spaces. In other words,  
it does not have an easy finite dimensional analogue. 
This is the point that makes L\'{e}vy Laplacian difficult and more interesting. 

In a paper of Ji-Saito \cite{JIS}, they proved that L\'{e}vy Laplacian can be identified with the Gross Laplacian of other infinite dimensional spaces. This theory is then extended to more general operators, which are called Exotic Laplacians, in Accardi-Ji-Saito \cite{ACJIS}. 

Our purpose in this paper is to investigate the relationships between these special spaces and the usual Hida distributions. We show that regular functionals on special spaces can be naturally imbedded into the usual Hida distributions. This imbedding enables us to calculate 
the problems on Exotic Laplacians as if they were the problems on Gross Laplacians. This makes the problems much easier, because we know quite well about Gross Laplacian. As an example, we  deal with heat equations generated by Exotic Laplacians in Section 4.

\section{Preliminaries}\label{2}
We follow the notations of \cite{ACJIS} and \cite{JIS}.
Let $\{ e_k \}_{k=1}^{\infty }$ be an orthonormal basis of a complex Hilbert space $H$, 
and let $\{ \lambda _k\}_{k=1}^{\infty} \subset \mathbb{R}$ satify 
\[
1< \lambda _1 \leq \lambda _2 \leq \cdots \text{and} 
\sum_{k=1}^{\infty} \lambda _k^{-2} < \infty .
\] 
For each $p \in \mathbb{R}$ and 
$\xi = \sum _{k=1}^{\infty}\alpha _k e_k \in H$, set 
\[
|\xi|_{p}^2 = \sum _{k=1}^{\infty} \lambda _k ^{2p} |\alpha _k|^2 .
\]
For each $ p \geq 0$, let $E_p = \{ \xi \in H ; |\xi|_{p} < \infty \}$ and let
$E_{-p}$ be the completion of $H$ with respect to $|\cdot |_{-p}$. 
A countably nuclear Hilbert space $E$ is defined by 
$E = \projlim _{p \to +\infty} E_p  $,
and its dual $E^{\ast}$ satisfies 
$E^{\ast} =\ilim_{p \to +\infty}E_{-p}$,
thus we have a basic Gelfand triple $E \subset H \subset E^{\ast}$. 

\subsection{Hida-Kubo-Takenaka spces}\label{21}
Let $\Gamma (E_p)$ denote the Fock space over $E_p$, i.e., 
\[
\Gamma (E_p) = \left\{ \phi = (f_n)_{n=0}^{\infty} ;
f_n \in E_p^{\hat{\otimes}n}, \| \phi \| _{H,p}^2 = 
\sum_{n=0}^{\infty} n! |f_n|_{E_p^{\hat{\otimes}n}}^2 < \infty \right\}.
\]
Identifying $\Gamma (E_0)$ with its dual space, we have 
\[
\cdots \subset \Gamma (E_q) \subset \Gamma (E_p) \subset \Gamma (E_0) \subset \Gamma (E_{-p}) \subset \Gamma (E_{-q}) \subset \cdots 
\] 
for $0<p<q$, and we have a higher Gelfand triple 
\[
(E) = \projlim_{p \to +\infty} \Gamma (E_{p}) \subset \Gamma (H) 
\subset (E)^{\ast} = \ilim_{p \to +\infty} \Gamma (E_{-p}).
\]
The {\em exponential vector} associated with $\xi \in E$ is defined by 
\[
\phi _{\xi} = \left( 1, \xi, \frac{\xi ^{\otimes 2}}{2!}, \cdots ,
\frac{\xi ^{\otimes n}}{n!} , \cdots \right) .
\]
It is easy to see that $\phi _ {\xi } \in (E) $. The {\em S-transform}  
of an element $\Phi \in (E)^{\ast}$ is defined by 
\[
S\Phi (\xi ) = \langle \! \langle \Phi, \phi _{\xi} \rangle \! \rangle , ~~~~~~\xi \in E,  
\]
where $\langle \! \langle \cdot , \cdot \rangle \! \rangle$ denotes the canonical $\mathbb{C}$-bilinear form on $(E)^{\ast} \times (E)$.

The trace operator $\tau$ is defined by 
\[
\tau = \sum _{k=1} ^{\infty } e_k^{\ast } J_e \otimes e_k^{\ast }, 
\]
which belongs to $E_{-1/2} \otimes E_{-1/2}$. 
And the Gross Laplacian $\Delta _G$ on $(E)$ is represented by 
\[
\Delta _G \phi = ( (n+2) (n+1) \tau \hat {\otimes} ^2 f_{n+2})
\] 
for any $\phi = (f_n)_{n=0} ^{\infty }$.

Let $a \in (1/2,\infty)$. 
Let $\text{Dom}(\Delta _{c,2a-1})$ denote the set of all $\Phi \in (E)^{\ast} $ such that the limit 
\[
\widetilde {\Delta} _{c,2a-1} S\Phi (\xi) = 
\lim_{N \to \infty} \frac {1}{N^{2a-1}} \sum _{k=1}^{N}
\langle (S \Phi)^{\prime \prime}(\xi ), e_k\otimes e_k \rangle , (\xi \in E)  
\]
exists for each $\xi \in E$ and the functional $\widetilde {\Delta}_{c,2a-1}(S\Phi)$ is the $S$-transform of an element  in $(E)^{\ast}$. The {\em Exotic Laplacian } $\Delta _{2a-1}$ on $\text {Dom} (\Delta _{c,2a-1})$ is defined by 
\[
\Delta _{c,2a-1}\Phi = S^{-1}(\widetilde {\Delta} _{c,2a-1}S \Phi ) . 
\]

\subsection{Exotic Hida-Kubo-Takenaka spaces}
Let a parameter $a \in (1/2, \infty )$ be fixed, 
and let a sequence $\{e_{a,k}\}_{k=1}^{\infty} \subset E^{\ast}$ satisfy the following three conditions:

\noindent
(C1) for each $k_1, k_2$, 
\begin{equation*}
\lim _{N \to \infty} \frac{1}{N^{2a-1}} \sum _{j=1} ^{N}
\overline{ \langle e_{a, k_1}, e_j \rangle }\langle e_{a,k_2}, e_j\rangle = \begin{cases}
1 & k_1 = k_2 \\
0 & k_1 \neq k_2, 
\end{cases}
\end{equation*}

\noindent
(C2) There exists some $p>0$ and $M>0$ satisfying 
\begin{equation*}
|e_{a,k}| _{-p} \leq M 
\end{equation*}
for all $k$.

\noindent
(C3) for any $ \alpha = \{ \alpha_k  \} _{k=1}^{\infty} \in \ell ^1 $ with $\alpha \neq 0$,  
\[
\sum \alpha _k e_{a,k} \neq 0, 
\]
where the sum is taken in $E^\ast$.

\begin{example}\label{pw}
Set 
\[
e_{a,k} = \sqrt{2a-1} \sum _{m=1}^{\infty} e^{i2\pi q_k r_m}m^{a-1} e_m , 
\]
where $\{ q_k\}_{k=1}^{\infty} = [0,1) \cap \mathbb{Q}$ and $\{r_m \}=\{ 0, 1, -1, 2, -2, \cdots \} $.

This example, which is taken from \cite{ACJIS}, satisfies (C1), (C2) and (C3).  
\end{example} 
\begin{proof}
Since 
\[
\sum _{j=1} ^{N}
\overline{ \langle e_{a, k_1}, e_j \rangle }\langle e_{a,k_2}, e_j\rangle = 
(2a-1) \sum _{j=1} ^N e^{i2\pi (q_{k_2}-q_{k_1}) r_j} j^{2a-2}, 
\]
it is easy to see that $\{e_{a,k}\}$ satisfies (C1). By the definition of $|\cdot |_{-p}$, 
\[
|e_{a,k}|_{-(2a-1)} ^2 = \sum _{m=1}^{\infty } \frac{1}{\lambda_m^2} \cdot \frac{m^{2a-2}}{\lambda _m ^{4a-4}} < + \infty ,
\]
which implies (C2). To see (C3), for each $\alpha \in \ell^1$, 
set a measure $\mu _{\alpha}$ on $[0,1)$ by 
\[
\mu _{\alpha} = \sum _{k=1} ^{\infty} \alpha _k \delta _{q_k }.
\]
Assume $\sum \alpha _k e_{a,k} = 0 $, then 
\[
\mu _\alpha (n) = \int _{[0,1)} e^{i2\pi nx} d\mu_{\alpha}(x) =  
\sum _{k=1} ^{\infty} \alpha _k e^{i2\pi q_k n} =0, 
\]
for any $n \in \mathbb{Z}$, because $\langle \sum \alpha _k e_{a,k}, e_m \rangle =0$ for all $m \in \mathbb{N}$. This shows $\mu _{\alpha }=0$, and hence $\alpha =0$.
\end{proof}

Let $H_{c,2a-1}$ denote the Hilbert space with orthonormal basis $\{e_{a,k}\}_{k=1}^{\infty} $, then by (C1), its inner product $\langle \cdot , \cdot \rangle _{c,2a-1}$ is characterized by 
\begin{equation}\label{CI}
\langle z,w\rangle _{c,2a-1} = \lim _{N \to \infty} \frac{1}{N^{2a-1}}
\sum_{k=1}^{N} \overline{ \langle z,e_k\rangle } \langle w,e_k \rangle 
\end{equation}
for all $z,w \in \text{Span} \{ e_{a,1}, e_{a,2}, \cdots  \}$. And by (C2), \eqref{CI} holds for all $z= \sum \alpha _k e_{a,k}$ and $w = \sum \beta _k e_{a,k}$, where $\alpha, \beta \in l^1$. 

Note that $H_{c, 2a-1}$ does not contain every $x \in E^{\ast}$ which has the limit 
$
\lim _{N \to \infty} \frac{1}{N^{2a-1}} \sum _{j=1} ^{N}
| \langle x, e_j \rangle |^2 .
$

Now let us formulate Hida-Kubo-Takenaka space. Let $\{ \lambda _{a,k}\}_{k=1}^{\infty} \subset \mathbb{R}$ satify 
\[
1< \lambda _{a,1} \leq \lambda _{a,2} \leq \cdots \text{and} 
\sum_{k=1}^{\infty} \lambda _{a,k}^{-2} < \infty ,
\] 
then by the same procedure as in Sec.\ref{2}, we have another basic Gelfand triple $\mathcal{N}_a \subset H_{c,2a-1} \subset \mathcal{N}_a^{\ast}$, called the {\em exotic triple} and 
the associated trace 
\[
\tau _a = \sum _{k=1}^{\infty } e_{a,k} \otimes e_{a,k} \in \mathcal{N}_{a,-1/2} \otimes \mathcal{N}_{a,-1/2}
\]
is called the {\em exotic trace} of order $2a-1$. 

By the same procedure as in Sec.\ref{21}, we have a chain of Fock spaces 
\[
\cdots \subset \Gamma (\mathcal{N}_{a,q}) \subset \Gamma (\mathcal{N}_{a,p}) \subset \Gamma (\mathcal{N}_{a,0}) \subset \Gamma (\mathcal{N}_{a,-p}) \subset \Gamma (\mathcal{N}_{a,-q}) \subset \cdots ,0<p<q, 
\]
and we have another higher Gelfand triple 
\[
(\mathcal{N}_a) = \projlim_{p \to +\infty} \Gamma (\mathcal{N}_{a,p}) \subset \Gamma (H_{c,2a-1}) 
\subset (\mathcal{N}_a)^{\ast} = \ilim_{p \to +\infty} \Gamma (\mathcal{N}_{a,-p}).
\]
which is called {\em exotic Hida-Kubo-Takenaka space} of order $2a-1$.  

The associated Gross Laplacian $\Delta _{G,2a-1} $ is defined on $( \mathcal{N}_a)$ and is represented by  
\[
\Delta _{G,2a-1} \phi = ( (n+2) (n+1) \tau _a \hat {\otimes} ^2 f_{n+2})
\] 
for any $\phi = (f_n)_{n=0} ^{\infty } \in (\mathcal{N}_a)$.

Note that in the case $a=1$, $Delia _{c, 1}$ is the L\'{e}vy Laplacian. 
This case is studied in \cite{JIS}. 

\section{Proof of the Main Theorem}
In this section, we prove that $\Gamma (\mathcal{N}_p)$ is imbedded into $(E)^{\ast}$ 
under conditions (C1)-(C3).  
\begin{lemma}\label{111}
Let $f_n$ be an element in 
$\mathcal{N}_{a,1}^{\otimes n} \cap E_{-p}^{\otimes n}$, 
then the inequality 
\[
|f_n|_{E_{-p}^{\otimes n}} \leq M^n
\left( \sum _{k = 1} ^{\infty} \lambda _{a,k} ^{-2}\right) ^{ n/2 }
|f_n|_{\mathcal{N}_{a,1}^{\otimes n}}
\]
holds.
\end{lemma}

\begin{proof}
Set 
\[
b_{k_1, k_2, \cdots ,k_n } = \langle 
e_{a, k_1} \otimes e_{a, k_2} \otimes \cdots \otimes e_{a, k_n}, f_n
\rangle _{c, 2a-1},
\]
then 
\[
f_n = \sum_{k_1, k_2, \cdots ,k_n =1} ^{\infty} 
b_{k_1, k_2, \cdots ,k_n } 
e_{a, k_1} \otimes e_{a, k_2} \otimes \cdots \otimes e_{a, k_n}, 
\]
where the convergence is defined in terms of $\mathcal{N}_{a,1}^{\otimes n}$. 
We have  
\[
\sum _{k_1, k_2, \cdots ,k_n =1} ^{\infty} 
|b_{k_1, k_2, \cdots ,k_n }|^2 \lambda _{a, k_1}^2 \lambda _{a, k_2}^2 \cdots \lambda _{a, k_n}^2 
= |f_n|_{\mathcal{N}_{a,1}^{\otimes n}}^2, 
\]
and, by Cauchy-Schwarz inequality, we obtain 
\begin{equation*}
\begin{split}
&\left( \sum _{k_1, k_2, \cdots ,k_n =1} ^{\infty} 
|b_{k_1, k_2, \cdots ,k_n }| \right)^2\\
&\leq 
\sum _{k_1, k_2, \cdots ,k_n =1} ^{\infty} 
|b_{k_1, k_2, \cdots ,k_n }|^2 \lambda _{a, k_1}^2 \lambda _{a, k_2}^2 \cdots \lambda _{a, k_n}^2 
\cdot 
\sum _{k_1, k_2, \cdots ,k_n =1} ^{\infty}
\lambda _{a, k_1}^{-2} \lambda _{a, k_2}^{-2} \cdots \lambda _{a, k_n}^{-2} \\
&=
\left( \sum _{k = 1} ^{\infty} \lambda _k ^{-2}\right) ^{n} 
|f_n|_{\mathcal{N}_{a,1}^{\otimes n}}^2 < \infty , 
\end{split}
\end{equation*}
which shows $\{ b_{k_1, k_2, \cdots ,k_n }\} \in \ell ^1 (\mathbb{N}^n)$. 
Now, by (C2), we obtain 
\begin{equation*}
\begin{split}
|f_n|_{E_{-p}^{\otimes n}}
&\leq
\sup  _{k_1, k_2, \cdots ,k_n}| e_{a,k_1} \otimes e_{a,k_2} \otimes \cdots \otimes e_{a,k_n}|_{E_{-p}^{\otimes n}} \cdot
\sum _{k_1, k_2, \cdots k_n =1} ^{\infty} 
|b_{k_1, k_2, \cdots ,k_n } |  \\
&\leq M^n
\left( \sum _{k = 1} ^{\infty} \lambda _{a, k} ^{-2}\right) ^{n/2} 
|f_n|_{\mathcal{N}_{a,1}^{\otimes n}} .
\end{split}
\end{equation*}
\end{proof}

More precisely, ``$f_n \in \mathcal{N}_{a,1}^{\otimes n} \cap E_{-p}^{\otimes n}$''
should imply the existance of a universal set $X$ satisfying $\mathcal{N}_{a,1}^{\otimes n} \subset X$ and $E_{-p}^{\otimes n} \subset X$.
In this case, however, it is not appriori given. 
The above lemma suggests a reasonable inclusion map $i : \mathcal{N}_{a,1}^{\otimes n} \to E_{-p}^{\otimes n}$ defined by 
\begin{equation*}
\begin{split}
\mathcal{N}_{a,1}^{\otimes n} \ni f_n &\mapsto \{ b_{k_1,k_2, \cdots ,k_n }\} \in \ell^1 \\
&\mapsto \sum_{k_1, k_2, \cdots k_n =1} ^{\infty} 
b_{k_1, k_2, \cdots ,k_n } 
e_{a, k_1} \otimes e_{a, k_2} \otimes \cdots \otimes e_{a, k_n} \in E_{-p}^{\otimes n}.
\end{split}
\end{equation*}
\begin{theorem}
$\mathcal{N}_{a,1}^{\otimes n}$ is imbedded into $E_{-p}^{\otimes n}$ by the above inclusion map $i$. 
\end{theorem}
\begin{proof}
What is left is to show that $i$ is injective. 

There is no need to show in the case $n=1$, because it is condition (C3). 
Assume that $i$ is injective when $n=m$. Let $\{ b_{k_1, \cdots , k_{m}, k_{m+1}} \}$ satisfy 
\[
\sum_{k_1, \cdots ,k_{m},k_{m+1}=1}^{\infty} b_{k_1,  \cdots , k_{m}, k_{m+1}}  e_{a,k_1}  \otimes \cdots \otimes e_{a,k_m} \otimes  e_{a,k_m+1} =0.
\]
By taking the right contraction with $e_j$, we have 
\begin{equation*} 
\begin{split}
\sum _{k_1, \cdots k_m =1}^{\infty} 
\left( \sum _{k_{m+1}=1}^{\infty} b_{k_1,  \cdots , k_m, k_{m+1}} \langle e_{a, k_{m+1}}, e_j\rangle \right) 
e_{a,k_1}  \otimes \cdots \otimes e_{a,k_m} 
=0
\end{split}
\end{equation*}
for all $j \in \mathbb{N}$. By the assumption, we obtain 
\[
\sum _{k_{m+1}=1}^{\infty} b_{k_1,  \cdots , k_m, k_{m+1}} \langle e_{a, k_{m+1}}, e_j\rangle
\]
for all $k_1, \cdots k_m$ and $j $. Now, by (C3), $b_{k_1,  \cdots , k_m, k_{m+1}} =0 $.
\end{proof}

\begin{corollary}
Let $\phi = (f_n) \in \Gamma (\mathcal{N}_{a,1})$, then $\phi$ is included in $\Gamma (E_{-q})$ for some $q \in \mathbb{N}$. 
\end{corollary}

\begin{proof}
Let $m \in \mathbb{N}$ satisfy $\lambda_1 ^m \geq M ( \sum _{i = 1} ^{\infty} \lambda _i ^{-2} )^{1/2}$, then by Lemma $\ref{111}$, we have 
\[
|f_n|_{E_{-(p+m)}^{\otimes n}} \leq 
|f_n|_{\mathcal{N}_{a,1}^{\otimes n}},
\]
and hence we obtain 
\[
\| \phi \| _{H,p} ^2 = \sum _{n=0} ^{\infty} 
n! |f_n|_{E_{-(p+m)}^{\otimes n}}^2
\leq 
\sum _{n=0} ^{\infty} 
n! |f_n|_{\mathcal{N}_{a,1}^{\otimes n}}^2
= \| \phi \| _{\mathcal{N}_{a,1}}^2 .
\]
\end{proof}

The results given in \cite{ACJIS} is rewritten in the following way. 
\begin{theorem}
Let $\phi = (f_n)_{n=1}^{\infty } \in (\mathcal {N}_a)$, then $i(\phi ) \in \text{Dom} (\Delta _{c,2a-1})$ 
and 
\[
i(\Delta _{G, 2a-1} \phi ) = \Delta _{c,2a-1} i(\phi) .
\]
\end{theorem}

\section{Exotic heat equations}
In this section, as an application of the results in the previous section, we consider 
the heat equation associated with the Exotic Laplacian:  
\begin{equation} \label{heat}
\begin{split}
\frac{\partial }{\partial t } u(t,\cdot ) &= \Delta _{c, 2a-1}u (t, \cdot) \\
u(0, \cdot )&= \Phi .
\end{split}
\end{equation}
We call $u : [0,T) \ni t \mapsto u(t, \cdot ) \in (E)^{\ast }$ is a solution of \eqref{heat} for $0 \leq t <T$ if: \\
~(i) $t \mapsto u(t, \cdot)$ is continuous on $[0,T)$ in the strong topology on $(E)^{\ast}$ \\
~(ii) $t \mapsto u(t, \cdot)$ is differentiable on $(0,T)$ in the strong topology on $(E)^{\ast}$ \\
~(iii) $u(t,\cdot ) \in \text{Dom}(\Delta _{c,2a-1})$ for any $t \in (0,T)$\\
and satisfies \eqref{heat}. 

It is difficult to find a solution for every initial condition $\Phi$. 
However, if $\Phi$ is regular in Exotic sense, we can apply the results on the heat equation associated with the Gross Laplacian to obtain the regularity of the solution in Exotic sense, then the solution can be  imbedded into the usual Hida-Kubo-Takenaka space.   

We use the following property on the Gross Laplacian. See e.g.  \cite{KUO1} for the proof. 

\begin{lemma}
Let $p>1$ satisfy 
\[
\lambda _{a,1}^{2(p-1)} >2. 
\] 
Then, for $\phi = (f_n) \in \Gamma (\mathcal{N}_{a,p} )$ and $0\leq t < \lambda _{a,1} ^{2(p-1)} / 
|\tau _a|_{\mathcal{N}_{a,-1} ^{\otimes2} }$, 
\[
P_{a,t} \phi = \left( \sum _{m=0} ^{\infty} \frac{(n+2m)!}{n!m!}
t^m (\tau _a^{\otimes m}\widehat {\otimes}_{2m}f_{n+2m} )\right)
\in \Gamma (\mathcal{N}_{a,1} ), 
\]
and $P_{a,t} \phi$ is the solution of 
\begin{equation*} 
\begin{split}
\frac{\partial }{\partial t } u(t,\cdot ) &= \Delta _{G, 2a-1}u (t, \cdot) \\
u(0, \cdot )&= \Phi ,
\end{split}
\end{equation*}
where the topology is that of $\Gamma (\mathcal{N}_{a,1} )$. 
\end{lemma}

Using this lemma and the relationship 
\[
i( \Delta _{G, 2a-1} \phi ) = \Delta _{c, 2a-1} i(\phi ) , \phi = (f_n)_{n=0}^{\infty} \in \Gamma (\mathcal{N}_{a,p} ) 
\]
for $p>1$, we obtain the following.  
\begin{theorem}
Let $\phi \in (E)^{\ast}$. Assume that there exists $\{e_{a,k}\}$ and $\{\lambda _{a,k} \}$ such that 
$\Phi = i(\phi ) $ with $\phi \in \Gamma (\mathcal{N}_{a,p})$, where $p>1$ satisfy
$\lambda _{a,1} ^{2(p-1)} > 2$. Then, $i (P_{a,t} \phi)$ is a solution of $\eqref{heat}$ 
for $0< t <\lambda _{a,1} ^{2(p-1)} / |\tau _a|_{\mathcal{N}_{a,-1} ^{\otimes 2}}$. 

In particular, if $\Phi = i(\phi)$ with $\phi \in (\mathcal{N}_a)$, 
then $i(P_{a,t}\phi)$ is a solution of $\eqref{heat}$ for $t>0$. 
\end{theorem}

\section*{Acknowledgements}
The author wishes to express his sincere gratitude to Prof. N. Obata and Prof. K. Saito for their helpful advices and comments.


\begin{thebibliography}{99}

\bibitem{ABO}
L. Accardi, A. Barhoumi and H. Ouerdiane,
\newblock {\em A quantum approach to Laplace operators}, 
\newblock Infinite Dimensional Analysis, Quantum Probability and Related Topics, {\bf 9} (2006) 215-248.

\bibitem{ACJIS}
L. Accardi, U. C. Ji and K. Sait\^{o},
\newblock {\em Exotic Laplacians and Associated Stochastic Processes},  
\newblock Infinite Dimensional Analysis, Quantum Probability and Related Topics, {\bf 12} (2009) 1-19.

\bibitem{GR3}
L. Gross, 
\newblock {\em Abstract Wiener spaces}, 
\newblock in ``Proceedings of the Fifth Berkeley Symposium on 
Mathematical Statistics and Probability'', Vol.II, Univ. of California 
Press, Berkeley, CA, 1967. 

\bibitem{GR1}
L. Gross, 
\newblock {\em Potential theory on Hilbert space},
\newblock J. Funct. Anal. {\bf 1} (1967) 123-181.

\bibitem{GR2}
L. Gross, 
\newblock {\em Abstract Wiener Measure and Infinite Dimensional 
Potential Theory}, 
\newblock Lecture Notes in Mathematics {\bf 140} (1970) 84-116, 
Springer. 

\bibitem{HIDA1}
T. Hida, H. H. Kuo, J. Potthoff, and L. Streit, 
\newblock {\em White Noise}, 
\newblock Kluwer, 1993. 

\bibitem{JIS}
U. C. Ji and K. Sait\^{o},
\newblock A similarity between the Gross Laplacian and the L\'{e}vy Laplacian, 
\newblock Infinite Dimensional Analysis, Quantum Probability and Related Topics, {\bf 10} (2007) 261-276.

\bibitem{KUO1}
H. H. Kuo, 
\newblock {\em White Noise Distribution Theory},
\newblock CRC Press, 1996.

\bibitem{OBA}
N. Obata, 
\newblock {\em White Noise Calculus and Fock Space}, 
\newblock Lecture Notes in Mathematics {\bf 1577} (1994) 
Springer. 

\bibitem{PL1}
P. L\'{e}vy, 
\newblock {\em Probl\`{e}mes Concrets d'Analyse Fonctionelle}, 
\newblock Gauthier-Villars, Paris, 1951. 

\end{thebibliography}
\end{document}